\documentclass[a4paper,10pt]{amsart}

\usepackage[all]{xy}
\usepackage[leqno]{amsmath}
\usepackage{amssymb}
\usepackage{amsthm}
\usepackage[french]{babel}
\usepackage[T1]{fontenc}
\usepackage[colorlinks=true]{hyperref}
\usepackage{amsxtra}
\usepackage{mathtools}
\usepackage{stmaryrd}

\newcounter{eq}[subsection]
\newcounter{seq}[subsubsection]

\theoremstyle{plain}
\newtheorem*{theo}{Th\'eor\`eme}
\newtheorem{theorem}[subsection]{Th\'eor\`eme}
\newtheorem{lemma}[subsection]{Lemme}
\newtheorem{prop}[subsection]{Proposition}
\newtheorem{cor}[subsection]{Corollaire}

\theoremstyle{definition}
\newtheorem{rem}[subsection]{Remarque}

\newcommand{\lra}{\longrightarrow}
\newcommand{\fact}{\mathrm{fact}}

\newcommand{\cE}{\mathcal{E}}
\newcommand{\cF}{\mathcal{F}}
\newcommand{\cG}{\mathcal{G}}
\newcommand{\cH}{\mathcal{H}}

\newcommand{\cL}{\mathcal{L}}
\newcommand{\cM}{\mathcal{M}}
\newcommand{\cN}{\mathcal{N}}
\newcommand{\cO}{\mathcal{O}}
\newcommand{\cP}{\mathcal{P}}

\newcommand{\ik}{\mathbf{k}}

\newcommand{\IQ}{\mathbf Q}
\newcommand{\IZ}{\mathbf Z}

\newcommand{\IA}{\mathbf A}
\newcommand{\IP}{\mathbf P}

\newcommand{\km}{\mathfrak{m}}

\newcommand{\PP}{\mathbf{P}}

\DeclareMathOperator{\Alb}{Alb}
\DeclareMathOperator{\Picard}{P}

\DeclareMathOperator{\Pic}{Pic}
\DeclareMathOperator{\NS}{NS}
\DeclareMathOperator{\Cl}{Cl}
\DeclareMathOperator{\Ca}{Ca}
\DeclareMathOperator{\codim}{codim}
\DeclareMathOperator{\Spec}{Spec}

\DeclareMathOperator{\Hom}{Hom}
\DeclareMathOperator{\id}{id}

\newcommand{\ie}{{i.e. }}

\title{Sur le produit de vari\'et\'es localement factorielles ou $\IQ$-factorielles}

\author{Samuel Boissi\`ere}

\address{Samuel Boissi\`ere,
Laboratoire de Math\'ematiques et Applications,
UMR 7348 du CNRS,
B\^atiment H3,
Boulevard Marie et Pierre Curie,
Site du Futuroscope,
TSA 61125,
86073 Poitiers Cedex 9,
France}
			
\email{samuel.boissiere@math.univ-poitiers.fr}

\urladdr{http://www-math.sp2mi.univ-poitiers.fr/$\sim$sboissie/}

\author{Ofer Gabber}

\address{Institut des Hautes \'Etudes Scientifiques, Le Bois-Marie,
35, route de Chartres, F-91440 Bures-sur-Yvette}

\email{gabber@ihes.fr}

\author{Olivier Serman}

\address{Laboratoire Paul Painlev\'e UMR CNRS 8524,
Universit\'e Lille 1, Cit\'e Scientifique, F-59655 Villeneuve-d'Ascq cedex}

\email{olivier.serman@math.univ-lille1.fr}

\urladdr{http://math.univ-lille1.fr/$\sim$serman/}

\date{\today}

\begin{document}

\begin{abstract}
Nous d\'emontrons que, sur un corps alg\'ebriquement clos, les lieux de factorialit\'e locale et de $\IQ$-factorialit\'e d'une vari\'et\'e sont ouverts, que le produit de vari\'et\'es 
localement factorielles est une vari\'et\'e localement factorielle, et que cette propri\'et\'e reste vraie pour les vari\'et\'es $\IQ$-factorielles 
lorsque le corps de base n'est pas la cl\^oture alg\'ebrique d'un corps fini.
\end{abstract}

\maketitle

\section*{Introduction}

Les propri\'et\'es de factorialit\'e et de $\IQ$-factorialit\'e locales d\'efinissent des vari\'et\'es alg\'ebriques poss\'edant des singularit\'es raisonnables: on 
sait par exemple que le lieu exceptionnel de toute r\'esolution d'une telle singularit\'e est de codimension pure \'egale \`a $1$. N\'eanmoins 
ces notions ne se comportent pas bien 
vis-\`a-vis de la topologie \'etale: il existe en effet des anneaux locaux factoriels dont le compl\'et\'e ne l'est plus. C'est ainsi le cas de la singularit\'e complexe 
d\'efinie par l'\'equation $xy+zw+z^3+w^3=0$ (voir~\cite[p. 104]{kawamata88}).
D'autres exemples de tels anneaux apparaissent lors de l'\'etude de certains espaces de modules de faisceaux (par exemple sur $\PP^2$, voir~\cite{drezet}).

Dans toute la suite nous appelons \og vari\'et\'e\fg~tout sch\'ema noeth\'erien int\`egre, s\'epar\'e, de type fini sur un corps alg\'ebriquement clos $\ik$.

Nous donnons une d\'emonstration des r\'esultats suivants:

\begin{theo}
Le lieu de factorialit\'e locale d'une vari\'et\'e est ouvert.
\end{theo}

\begin{theo}
Soient $X$ et $Y$ deux vari\'et\'es localement factorielles. Le produit $X \times_\ik Y$ est une vari\'et\'e localement factorielle.
\end{theo}

Ce bon comportement de la factorialit\'e locale est assez inattendu au vu des probl\`emes rappel\'es plus haut. Il semblait m\^eme n\'ecessaire de recourir \`a la notion d'anneau local 
\textit{g\'eom\'etriquement factoriel}
pour esp\'erer obtenir de tels \'enonc\'es (cf~\cite[Expos\'e XIII]{sga2}).

Ces r\'esultats s'\'etendent aux vari\'et\'es localement $\IQ$-factorielles, moyennant pour le second \'enonc\'e une hypoth\`ese 
suppl\'ementaire sur le corps de base $\ik$.

\begin{theo}
 Le lieu de $\IQ$-factorialit\'e d'une vari\'et\'e est ouvert.
\end{theo}

\begin{theo}
Soient $X$ et $Y$ deux vari\'et\'es localement $\IQ$-factorielles. Si $\ik$ n'est pas la cl\^oture alg\'ebrique d'un corps fini, alors le produit $X \times_\ik Y$ est une 
vari\'et\'e localement $\IQ$-factorielle.
\end{theo}

Ce dernier th\'eor\`eme permet en particulier de lever une hypoth\`ese superflue dans le r\'esultat de Fu et Namikawa~\cite[Theorem 2.2]{funamikawa} portant sur l'unicit\'e de
 la r\'esolution cr\'epante d'un 
produit de vari\'et\'es localement $\IQ$-factorielles singuli\`eres admettant une r\'esolution cr\'epante dont le diviseur exceptionnel est irr\'eductible.

Bingener et Flenner ont donn\'e dans~\cite{bingenerflenner} une d\'emonstration de l'ouverture des lieux factoriel et $\IQ$-factoriel d'un sch\'ema normal excellent admettant une 
r\'esolution des singularit\'es. Leur d\'emonstration est de nature locale. Nous donnons ici une d\'emonstration valable pour toute vari\'et\'e d\'efinie sur un corps 
alg\'ebriquement clos. En ce qui concerne 
la question de la stabilit\'e par produit direct, il r\'esulte par exemple de Boutot~\cite[III.2.14]{boutot} que le produit d'une vari\'et\'e localement g\'eom\'etriquement factorielle et d'une vari\'et\'e lisse est localement g\'eom\'etriquement 
factorielle. 

Une vari\'et\'e normale est localement factorielle si et seulement si les notions de diviseurs de Weil et de Cartier co\"incident, ce qui revient 
\`a dire que tout faisceau
divisoriel (voir~\S\ref{faisceaudivisoriel}) est inversible. On est ainsi amen\'e pour montrer l'ouverture des lieux factoriel ou $\IQ$-factoriel d'une vari\'et\'e 
normale \`a 
\'etudier le lieu 
d'inversibilit\'e des faisceaux divisoriels d\'efinis sur cette vari\'et\'e. Nous \'etablissons ces r\'esultats en nous appuyant sur diverses propri\'et\'es des vari\'et\'es 
d'Albanese, qui sont rappel\'ees dans la section~\ref{albanese}. Les
 deux th\'eor\`emes concernant la factorialit\'e locale sont d\'emontr\'es dans les sections~\ref{factouvert} et~\ref{factprod}, et leurs extensions \`a la 
$\IQ$-factorialit\'e font l'objet de la section~\ref{Qfact}.

\section{Factorialit\'e et  $\IQ$-factorialit\'e}\label{section:def}

\subsection{}

Soit $X$ une vari\'et\'e normale. Notons $\Cl(X)$ le groupe des classes de 
diviseurs de Weil sur $X$ et $\Ca(X)$ le groupe des classes de diviseurs de Cartier sur $X$. Le morphisme de groupes $\Ca(X)\to\Cl(X)$ est injectif d'image 
l'ensemble des classes de diviseurs de Weil localement principaux. 
Rappelons que $X$ est dite \emph{localement factorielle} si tous ses anneaux locaux sont factoriels, autrement dit si {${\Ca(X) \to \Cl(X)}$} est un isomorphisme, 
et \emph{localement 
$\IQ$-factorielle} si tous ses anneaux locaux sont $\IQ$-factoriels (i.e. si le groupe des classes $\Cl(\cO_{X,x})$ est de torsion pour tout point $x$), autrement dit si le quotient 
$\Cl(X)/\Ca(X)$ est de torsion.

\subsection{}

Avec les m\^emes hypoth\`eses, notons $\Pic(X)$ le groupe de Picard de $X$, naturellement isomorphe \`a $\Ca(X)$ puisque $X$ est int\`egre. 
Rappelons que si $U\subset X$ est l'ouvert compl\'ementaire d'un ferm\'e de codimension sup\'erieure ou \'egale \`a $2$, le morphisme de restriction $\Cl(X)\to\Cl(U)$ 
est un isomorphisme. Rappelons le lemme bien connu:

\begin{lemma}\label{lemm:factpic}
Soient $X$ une vari\'et\'e normale, et $U$ son ouvert de lissit\'e.
Alors $X$ est localement factorielle si et seulement si le morphisme de restriction ${\Pic(X)\to\Pic(U)}$ est un isomorphisme.
\end{lemma}

\section{Faisceaux r\'eflexifs}

\subsection{}

Soit $\cF$ un faisceau coh\'erent sur un sch\'ema noeth\'erien $X$. Rappelons que:
\begin{itemize}
\item $\cF$ est dit \emph{r\'eflexif} si l'application naturelle $\cF\to\cF^{\vee\vee}$ est un isomorphisme,
\item $\cF$ est dit \emph{normal} si pour tout ouvert $V\subset X$ et tout ferm\'e $Y\subset V$ de codimension au moins deux, 
la restriction $\cF(V)\to\cF(V\setminus Y)$ est bijective.
\end{itemize}

Lorsque $X$ est normal, un faisceau r\'eflexif de rang $1$ en tout point g\'en\'erique d'une composante irr\'eductible est appel\'e \emph{faisceau divisoriel}.  

Nous utiliserons les deux caract\'erisations suivantes de la r\'eflexivit\'e:

\begin{prop}{\cite[Proposition 1.1]{hartshorne2}}\label{prop:Exact}
Un faisceau coh\'erent $\cF$ sur un sch\'ema noeth\'erien et int\`egre $X$ est r\'eflexif si et seulement s'il s'ins\`ere, au moins localement, dans une suite exacte :
\[
0\to\cF\to\cE\to\cG\to 0,
\]
o\`u $\cE$ est localement libre et $\cG$ est sans torsion.
\end{prop}

\begin{prop}{\cite[Proposition 1.6]{hartshorne2}}\label{prop:Refl}
Soit $\cF$ un faisceau coh\'erent sur un sch\'ema noeth\'erien int\`egre et normal. Les conditions suivantes sont \'equivalentes:
\begin{enumerate}
\item\label{prop:Refl1} $\cF$ est r\'eflexif;
\item\label{prop:Refl2} $\cF$ est sans torsion et normal;
\item\label{prop:Refl3} $\cF$ est sans torsion, et pour tout ouvert $W\subset X$ et tout ferm\'e $Y\subset W$ de codimension au moins deux, $\cF_W\cong j_*\cF_{W\setminus Y}$, o\`u $j\colon W\setminus Y\hookrightarrow W$ est l'immersion.
\end{enumerate}
\end{prop}

\begin{lemma}\label{lemm:refl}
Soient $X$ une vari\'et\'e normale et $j\colon V\hookrightarrow X$ l'immersion d'un ouvert tel que $\codim_X(X\setminus V)\geq 2$. Pour tout faisceau localement libre non nul $\cE$ sur $V$, $j_*\cE$ est un faisceau r\'eflexif sur $X$.
\end{lemma}

\begin{proof}
La coh\'erence de $j_*\cE$ r\'esulte de~\cite[Th\'eor\`eme 9]{serre} (ou~\cite[VIII, Proposition 3.2]{sga2}) puisque $\codim_X(X\setminus V)\geq 2$. 
Il est clair que $j_*\cE$ est sans torsion et l'on v\'erifie ais\'ement qu'il est normal: 
la restriction $j_*\cE(W)\to j_*\cE(W\setminus Y)$ est par d\'efinition $\cE(W\cap V)\to\cE((W\setminus Y)\cap V)=\cE((W\cap V)\setminus (Y\cap V))$ et $Y\cap V$ est 
de codimension au moins deux dans $V$. Puisque $\cE$ est localement libre, il est r\'eflexif, donc normal en utilisant la proposition~\ref{prop:Refl}, et cette restriction est bijective.
\end{proof}

\subsection{} \label{faisceaudivisoriel}

Si $X$ est une vari\'et\'e normale et $i\colon U\hookrightarrow X$ l'immersion de son ouvert de lissit\'e, ce qui pr\'ec\`ede 
montre qu'on peut d\'efinir pour tout diviseur de Weil de $D$ un faisceau divisoriel $\cO_X(D)=i_\ast\cO_U(D_{|U})$: c'est 
l'\emph{extension r\'eflexive} du faisceau inversible sur $U$ associ\'e \`a $D$. Pour toute immersion 
$j\colon V\hookrightarrow~X$ d'un ouvert $V$ contenant $U$ et tout diviseur de Cartier $E$ sur $V$, si $D$ est une extension de $E$ \`a $X$ comme diviseur de Weil, la 
proposition~\ref{prop:Refl} montre que $\cO_X(D)\cong j_*\cO_V(E)$ puisque ces deux faisceaux sont \'egaux sur $U$ et sont r\'eflexifs. Le ferm\'e de $X$ sur lequel $\cO_X(D)$ n'est pas localement libre est donc exactement le lieu o\`u $D$ n'est pas localement principal (\ie n'est pas de Cartier).

\begin{lemma}
\label{lem:ramero}
Soient $S$ un sch\'ema, $X$ un sch\'ema noeth\'erien et int\`egre sur $S$ et $\cF$ un faisceau coh\'erent et sans torsion sur $X$. Soient $s$ un point de $S$ non n\'ecessairement ferm\'e et $x$ un point
 de $X$ sur $k(s)$. Alors le faisceau $\cF_s$ sur la fibre $X_s$ est inversible au point de $X_s$ d\'efini par $x$ si et seulement si $\cF$ l'est au point $x$.
\end{lemma}

\begin{proof}
Supposons que $\cF_s$ est localement libre de rang $1$ en $x$. Notons $M$ le $\cO_{X,x}$-module $\cF_x$ et $\km_s$ l'id\'eal maximal de $S$ en $s$. Le $\cO_{X_s,x}$-module $(\cF_s)_x$ 
vaut  $(\cF_s)_x=M\otimes_{\cO_{S,s}}k(s)\cong M/\km_sM$ et est localement libre de rang $1$. Soit $g$ un \'el\'ement de $M$ dont la classe engendre $M/\km_sM$: alors $M=g\cO_{X,x}+\km_sM$. Puisque l'id\'eal $\km_s\cO_{X,x}$ est inclus dans l'id\'eal maximal $\km_x$ de $\cO_{X,x}$, le lemme de Nakayama assure que $M$ est engendr\'e par $g$. Comme il est sans torsion, c'est un module libre de rang $1$. L'implication r\'eciproque est imm\'ediate.
\end{proof}

En particulier, si $\cF$ est un faisceau r\'eflexif de rang $1$ sur le produit $X\times_\ik Y$ de deux vari\'et\'es $X$ et $Y$, et 
si pour un point ferm\'e $x$ de $X$ le faisceau $\cF_{|\{x\}\times Y}$ sur $Y$ est inversible, alors $\cF$ est inversible en tout point de $\{x\}\times Y$.

\section{Vari\'et\'es d'Albanese et de Picard} \label{albanese}

\subsection{} Soit $X$ une $\ik$-vari\'et\'e projective normale. Nous rappellons dans cette section la construction d'une vari\'et\'e ab\'elienne $\Picard(X)$ qui param\`etre 
l'ensemble $\Cl_a(X)$ des 
classes d'\'equivalence lin\'eaire de diviseurs de Weil sur~$X$ alg\'ebriquement \'equivalents \`a z\'ero. Il est bien connu que la sous-vari\'et\'e r\'eduite associ\'ee 
au sch\'ema de Picard $\Pic^0(X)$ param\'etrant les faisceaux inversibles alg\'ebriquement \'equivalent \`a z\'ero sur $X$ 
est une vari\'et\'e ab\'elienne dont la vari\'et\'e duale n'est autre que la \textit{vari\'et\'e d'Albanese morphique} $\Alb_m(X)$ de $X$, qui factorise tout morphisme 
point\'e de $X$ vers une vari\'et\'e ab\'elienne. Il est donc naturel de consid\'erer la \textit{vari\'et\'e d'Albanese} $\Alb(X)$ de $X$ jouissant de la m\^eme propri\'et\'e 
pour les applications rationnelles (d\'efinie par exemple dans~\cite{serre}). Nous montrons, suivant Lang~\cite{lang}, que sa vari\'et\'e duale $\Picard(X)$ param\`etre les classes 
d'\'equivalence de diviseurs de Weil alg\'ebriquement \'equivalents 
\`a z\'ero.

Rappelons plus pr\'ecis\'ement la d\'efinition de la vari\'et\'e d'Albanese. Soit $a$ un $\ik$-point lisse de $X$. L'application d'Albanese $\alpha_X \colon X \dashrightarrow \Alb(X)$ 
est l'application rationnelle point\'ee universelle pour les applications rationnelles de $X$ vers une $\ik$-vari\'et\'e ab\'elienne envoyant $a$ sur $0$. L'ouvert de 
d\'efinition $V$ de cette application contient le lieu lisse $U$ de $X$.

La vari\'et\'e $\Picard (X)=\Alb(X)^\vee$ est une vari\'et\'e de Picard au sens de Lang~\cite[IV \S 4]{lang}. 
Nous donnons ici une pr\'esentation moderne de ce r\'esultat, en montrant que $\Picard(X)$ repr\'esente, dans la cat\'egorie des $\ik$-sch\'emas normaux de type fini et r\'eduits, le foncteur
 $P^0_X$ associant \`a tout $\ik$-sch\'ema $S$ le groupe des classes d'isomorphisme de faisceaux 
inversibles $\cL$ sur $U \times S$ munis d'une trivialisation $u\colon\cO_{\{a\}\times S}\buildrel\sim\over\to\cL_{|\{a\}\times S}$ au-dessus de $\{a\}\times S$ dont les restrictions aux fibres g\'eom\'etriques de la projection $U \times S \to S$ sont alg\'ebriquement \'equivalentes \`a z\'ero.

Notons $\cP$ le fibr\'e de Poincar\'e sur $\Alb(X) \times \Picard(X)$, et $\cE$ son image inverse sur $U_X \times \Picard(X)$. Ce fibr\'e d\'efinit une 
transformation naturelle de $\Picard(X)$ vers~$P^0_X$. 

\begin{prop}\label{picardWeil}
 La vari\'et\'e ab\'elienne $\Picard(X)$ repr\'esente la restriction du foncteur $P^0_X$ \`a la cat\'egorie des $\ik$-sch\'emas de type fini normaux.
\end{prop}

\begin{proof} Montrons d'abord que la transformation naturelle $\varphi \colon \Picard(X)\lra P^0_X$ associ\'ee \`a $\cE$ d\'efinit, pour toute extension alg\'ebriquement close $K$ de $\ik$, 
un isomorphisme de $\Picard(X)(K)$ sur $P^0_X(K)$. 

Rappelons que la vari\'et\'e d'Albanese $\Alb(X_K)$ de 
$X_K$ est isomorphe \`a $\Alb(X)_K$, o\`u $Y_K$ d\'esigne le $K$-sch\'ema $Y \times_{\Spec \ik} {\Spec K}$ d\'eduit par extension des scalaires d'un $\ik$-sch\'ema $Y$ (voir~\cite[Lemma 2.3]{gabber}).
Le morphisme $\varphi(K)$ s'identifie alors \`a ${(\alpha_X)_K}^\ast \colon \Pic^0(\Alb(X)_K) \to P^0_X(K) \subset \Pic(U_K)$.

L'injectivit\'e de ${(\alpha_X)_K}^\ast$ s'obtient par r\'eduction au cas des courbes. D'apr\`es~\cite[Proposition 2.4]{gabber}, il existe une courbe $C$ projective lisse sur $K$ contenue dans $U_K$ (passant par $a$) telle que l'application $\Alb(C) \to \Alb(U_K)$ est surjective et de 
noyau lisse et connexe. Le morphisme $\Pic^0(\Alb(U)) \to \Pic^0(\Alb(C))$ est donc injectif, et la conclusion r\'esulte de l'isomorphisme 
$\Pic^0(\Alb(C)) \simeq \Pic^0(C)$.

Montrons maintenant la surjectivit\'e. Soit $L$ un faisceau inversible sur $U_K$ alg\'ebriquement \'equivalent \`a z\'ero. Il existe une courbe $\Gamma$ projective lisse irr\'eductible sur $K$, un faisceau inversible $\cL$ sur 
$U_K \times \Gamma$ et deux points $t$ et $t'$ de $\Gamma(K)$ tels que $\cL_t$ soit isomorphe \`a $L$ et $\cL_{t'}$ \`a $\cO_{U_K}$. Le faisceau 
$\cL \otimes p_\Gamma^\ast \cL_a^{-1}$ d\'efinit un morphisme de $U_K$ vers la jacobienne $J_\Gamma$ de $\Gamma$ envoyant 
$a$ sur $0$, qui induit un morphisme $f\colon \Alb(X)_K \lra J_\Gamma$. Le faisceau $L$ s'obtient alors comme image inverse via $(\alpha_X)_K$ de $f^\ast ({\cP_\Gamma})_t$, 
o\`u $\cP_\Gamma$ d\'esigne le fibr\'e de Poincar\'e 
sur $J_\Gamma \times \Gamma$ v\'erifiant $({\cP_\Gamma})_{t'} \simeq  \cO_{J_\Gamma}$.

Pour d\'efinir la transformation inverse $\psi \colon P^0_X \lra \Picard(X)$, on consid\`ere \`a nouveau la jacobienne de la courbe projective lisse $C$ consid\'er\'ee plus haut. La vari\'et\'e 
$\Picard(X)$ se plonge donc dans la jacobienne $J_C=\Alb(C)^\vee$ de $C$.

Soient $S$ un sch\'ema normal de type fini, et $\cL \in P_X^0(S)$ une famille de faisceaux inversibles sur $U$ alg\'ebriquement \'equivalents \`a z\'ero param\'etr\'ee par $S$ et 
munie d'un isomorphisme $\cO_S \buildrel\sim\over\to\cL_{|\{a\}\times S}$. Sa restriction 
$\cL_{|C \times S}$ d\'efinit un \'el\'ement de $P_C^0(S)$, donc 
un morphisme $S \to J_C$. Comme $S$ est r\'eduit et que $K$ est alg\'ebriquement clos, ce morphisme se factorise \`a travers $\Picard(X) \subset J_C$. Soit $c_\cL$ le morphisme $S \to \Picard(X)$ ainsi obtenu.

V\'erifions que $(\id \times c_\cL)^\ast\cE$ est isomorphe \`a $\cL$. La restriction \`a $C \times \{s\}$ du faisceau inversible 
$\cM=\cL^{-1} \otimes (\id \times c_\cL)^\ast\cE$ est triviale pour tout $s \in S$. Soit $\eta$ un point maximal de $S$ et $\overline{k(\eta)}$ une cl\^oture alg\'ebrique de $k(\eta)$. Puisque le morphisme de restriction $P^0_X(\overline{k(\eta)}) \to P^0_C(\overline{k(\eta)})$, qui s'identifie \`a 
$\Picard(X)(\overline{k(\eta)}) \to J_C(\overline{k(\eta)})$, est injectif, 
le tir\'e en arri\`ere sur $U_{\overline{k(\eta)}}$ du faisceau $\cM_\eta\coloneqq\cM_{|U\times\{\eta\}}$ est trivial, donc l'espace des sections de $\cM_\eta$ est de dimension un et fournit une trivialisation de $\cM_\eta$.

Puisque $S$ est normal, il r\'esulte alors de~\cite[21.4.13 Err${}_{\text{IV}}$ 53]{egaIV4} que $\cM$ provient d'un faisceau inversible $\cN$ sur $S$, qui est n\'ecessairement trivial.
\end{proof}

\begin{rem}\label{directproof}
 Le fibr\'e de Poincar\'e sur le produit $X \times \Pic^0(X)_{\mathrm{red}}$ convenablement rigidifi\'e d\'efinit une application $\Pic^0(X)_{\mathrm{red}} \to \Picard(X)$. 
On obtient par dualit\'e une 
application $\Alb(X) \to \Alb_m(X)$, qui est pr\'ecis\'ement l'application associ\'ee au morphisme universel $X \to \Alb_m(X)$. On en d\'eduit aussit\^ot que l'application 
d'Albanese $X \dashrightarrow \Alb(X)$ est partout d\'efinie 
si et seulement si tout diviseur de Weil alg\'ebriquement \'equivalent \`a z\'ero est Cartier, donc en particulier si $X$ est localement factorielle.

Lorsque le corps de base $\ik$ n'est pas la cl\^oture alg\'ebrique d'un corps fini, ce r\'esultat s'\'etend aux vari\'et\'es localement $\IQ$-factorielles. 
Il r\'esulte en effet de~\cite[Th\'eor\`eme 6]{serre} que le noyau du morphisme $\Alb(X) \to \Alb_m(X)$ est lisse et connexe. S'il est non nul, on en d\'eduit par dualit\'e que
 le quotient $\Picard(X)/\Pic^0(X)_\mathrm{red}$ est une vari\'et\'e ab\'elienne non triviale, qui contient d'apr\`es~\cite[Theorem 10.1]{freyjarden}
un $\ik$-point de non--torsion. Un tel point d\'efinit un diviseur de Weil sur $X$ qui n'est pas $\IQ$-Cartier: l'application d'Albanese ${X \to \Alb(X)}$ est donc partout 
d\'efinie si $X$ est 
une vari\'et\'e projective localement $\IQ$-factorielle d\'efinie sur un corps alg\'ebriquement clos qui n'est pas la cl\^oture alg\'ebrique d'un corps fini.

Remarquons qu'il existe des vari\'et\'es projectives non localement factorielles dont l'application d'Albanese est partout d\'efinie. C'est le cas des vari\'et\'es 
unirationnelles non localement factorielles, telles que les espaces 
projectifs \`a poids singuliers, ou encore les espaces de modules de fibr\'es $G$-principaux semi-stables sur une courbe projective de 
genre $g \geqslant 2$, pour tout groupe
semi-simple simplement connexe $G$ qui n'est pas sp\'ecial au sens de Serre (voir~\cite{BLS}).

Plus g\'en\'eralement, on obtient facilement des vari\'et\'es projectives non localement $\IQ$-factorielles dont l'application d'Albanese est partout d\'efinie en 
consid\'erant certaines 
vari\'et\'es toriques (voir~\cite[\S3]{fulton}).
\end{rem}

\subsection{}

Nous appelons ici \emph{groupe de N\'eron--Severi} de $X$ le groupe des classes d'\'equivalence alg\'ebrique de diviseurs de Weil 
sur $X$, not\'e $\NS(X)=\Cl(X)/\Cl_a(X)$. Puisque $X$ est normale et projective sur $\ik$, $\NS(X)$ est un groupe de type fini, par exemple d'apr\`es~\cite[Th\'eor\`eme 3]{kahn}.

Dans le cas o\`u $\ik$ n'est pas la cl\^oture alg\'ebrique d'un corps fini, on en d\'eduit:

\begin{cor}\label{cor:m-fact}
Une vari\'et\'e projective localement $\IQ$-factorielle d\'efinie sur un corps alg\'ebriquement clos qui n'est pas la cl\^oture alg\'ebrique d'un corps fini est 
$m$-factorielle pour un certain $m>0$.
\end{cor}

\begin{proof}
 D'apr\`es la remarque~\ref{directproof}, tout diviseur de Weil sur $X$ alg\'ebriquement \'equivalent \`a z\'ero est Cartier. Puisque le groupe de N\'eron--Severi $\NS(X)$ est de type 
fini, il existe un entier $m>0$ tel que tout diviseur de Weil est $m$-Cartier.
\end{proof}

\section{Ouverture de la factorialit\'e locale}\label{factouvert}

Nous notons $X^\fact$ le \emph{lieu de factorialit\'e locale} d'un sch\'ema $X$, d\'efini comme l'ensemble des points $x\in X$ tels que $\cO_{X,x}$ est factoriel. Rappelons que sur un sch\'ema noeth\'erien
int\`egre et localement factoriel, tout faisceau coh\'erent r\'eflexif de rang un est inversible (voir \cite[Proposition 1.9]{hartshorne2}). Plus g\'en\'eralement:

\begin{lemma}\label{lem:extension} Soient $X$ une vari\'et\'e normale et $U$ son ouvert de lissit\'e. Le lieu de factorialit\'e locale de $X$ est le lieu sur lequel l'extension 
r\'eflexive de tout faisceau inversible sur $U$ reste localement libre.
\end{lemma}

\begin{theorem}\label{theo:gros1}
Le lieu de factorialit\'e locale d'une vari\'et\'e est ouvert.
\end{theorem}

\begin{proof}
Soit $X$ une vari\'et\'e d\'efinie sur un corps $\ik$. Puisque le lieu des points o\`u $X$ est normale est ouvert d'apr\`es~\cite[\S 6.13]{egaIV2}, 
on peut supposer $X$ normale. La question \'etant locale, on peut supposer que $X$ est affine. Soit $\overline{X}$ l'adh\'erence de $X$ dans un espace projectif. 
Puisque $X$ est normale, elle se plonge comme un ouvert dans la normalis\'ee $\overline{X}^{\nu}$ de $\overline{X}$. 
On peut donc finalement supposer que $X$ est projective sur $\ik$.
\par{\it Premi\`ere \'etape. }
Soient $U$ l'ouvert de lissit\'e de $X$, $a$ un point de $U$ fix\'e et ${\alpha_X\colon (X,a)\dashrightarrow (\Alb(X),0)}$ l'application point\'ee universelle. 
Notons $V$ son ouvert de d\'efinition, qui contient $U$, et $j\colon V\hookrightarrow~X$ l'immersion. Soit $\cP$ le fibr\'e de Poincar\'e sur $\Alb(X)\times\Picard(X)$. 
Consi\-d{\'e}rons les morphismes:
\[
\xymatrix{V\times\Picard(X)\ar[r]^-{\alpha_X\times\id}\ar[d]_{j\times\id} & \Alb(X)\times\Picard(X)\\X\times\Picard(X)}
\]
et posons:
\begin{align*}
\cE&=(\alpha_X\times\id)^*\cP\in\Pic(V\times\Picard(X)),\\
\cL&=(j\times\id)_*\cE.
\end{align*}
D'apr\`es le lemme~\ref{lemm:refl}, $\cL$ est un faisceau r\'eflexif.

Soient $W$ l'ouvert de $X\times\Picard(X)$ sur lequel $\cL$ est localement libre et $Z$ son compl\'ementaire. Par la projection $\pi\colon X\times\Picard(X)\to X$, 
l'image $W_0=\pi(W)$ est un ouvert de $X$ ; posons aussi $Z_0=\pi(Z)$.
Le ferm\'e $Z$ ne contient aucune composante irr\'eductible d'une fibre au-dessus d'un point de $W_0$: en effet, ces fibres sont irr\'eductibles et l'intersection avec $W$ 
y d\'efinit des ouverts non vides.
 Au point g\'en\'erique $\eta$ de $W_0$, la restriction $\cL_\eta$ de $\cL$
\`a $\{\eta\}\times\Picard(X)$ est un faisceau r\'eflexif car le changement de base $\{\eta\} \hookrightarrow
W_0$ est plat: c'est donc un faisceau inversible en codimension $1$.
D'apr\`es le lemme~\ref{lem:ramero}, la trace de $Z$ sur la fibre en $\eta$
est donc de codimension au moins $2$.
Le th\'eor\`eme de Ramanujam-Samuel~\cite[21.14.3 (ii)]{egaIV4} (voir aussi~\cite[III.1.6]{boutot}) s'applique alors, concluant que la paire $(W_0\times\Picard(X),Z\cap(W_0\times\Picard(X)))$ est parafactorielle. 
Il en r\'esulte que l'application de restriction:
\[
\Pic(W_0\times\Picard(X))\to\Pic((W_0\times\Picard(X))\setminus (Z\cap(W_0\times\Picard(X))))
\]
est surjective. Puisque 
$\cL$ est localement libre sur $(W_0\times\Picard(X))\setminus Z\subset W$, il d\'efinit un faisceau inversible sur $W_0\times\Picard(X)$,
 qui n'est autre que la restriction de $\cL$ \`a cet ouvert d'apr\`es la proposition~\ref{prop:Refl}. Cela montre que $Z_0=X\setminus W_0$.
Autrement dit, le lieu $Z$ o\`u $\cL$ n'est pas localement libre est l'image r\'eciproque du ferm\'e $Z_0$ de $X$ par la projection $\pi$, soit $Z=Z_0\times\Picard(X)$, et on a $Z_0\subset X\setminus V$.

R\'eciproquement, le fibr\'e $\cL$ \'etant localement libre sur $(X\setminus Z_0)\times \Picard(X)$, par propri\'et\'e universelle de $\widehat{\Picard(X)}$, il existe un unique morphisme:
\[
(X\setminus Z_0)\xrightarrow{f} \widehat{\Picard(X)}=\Alb(X)
\]
tel que $f^*\cP\cong \cL_{|(X\setminus Z_0)\times\Picard(X)}$. Par la propri\'et\'e universelle de la vari\'et\'e d'Albanese, $\alpha_X$ est donc d\'efini au moins sur 
$X\setminus Z_0$, donc $X\setminus Z_0\subset V$. Finalement on a $Z_0=X\setminus V$ (\ie $W_0=V$), donc l'ouvert de libert\'e locale de $\cL$ est exactement $V\times\Picard(X)$.

\par{\it Deuxi\`eme \'etape. }
Puisque $\cL$ est r\'eflexif, on peut recouvrir $X \times \Picard(X)$ par une famille d'ouverts $W_1,\ldots,W_n$ au-dessus desquels la restriction de $\cL$ s'ins\`ere dans une suite exacte de faisceaux coh\'erents :
\[
0\to\cL_{|W_i}\to\cF_{W_i}\to\cG_{W_i}\to 0,
\]
o\`u $\cF_{W_i}$ est localement libre sur $W_i$ et $\cG_{W_i}$ sans torsion (d'apr\`es la proposition~\ref{prop:Exact}). 
Puisque $\cG_{W_i}$ est sans torsion, quitte \`a choisir $W_i$ plus petit on peut supposer que ce faisceau s'ins\`ere dans une suite exacte de faisceaux coh\'erents:
\[
0\to\cG_{W_i}\to\cH_{W_i}\to\cN_{W_i}\to 0,
\]
o\`u $\cH_{W_i}$ est localement libre.
Il existe un ouvert non vide $\Omega$ 
de $\Picard(X)$ au-dessus duquel chacun des faisceaux $\cN_{W_i}$ est plat.
Pour tout diviseur de Weil $D$ sur $X$ alg\'ebriquement \'equivalent \`a z\'ero et tel que $[D]\in\Omega$, on obtient alors 
 sur chaque fibre $(X \times [D]) \cap W_i$ des suites exactes :
\[
0 \to \cL_{|(X\times [D]) \cap W_i} \to (\cF_{W_i})_{|(X\times [D]) \cap W_i} \to (\cG_{W_i})_{|(X\times [D]) \cap W_i} \to 0,
\]
o\`u $(\cF_{W_i})_{|(X\times[D])\cap W_i}$ est localement libre et $(\cG_{W_i})_{|(X\times[D])\cap W_i}$ sans torsion. Le faisceau $\cL_{|X\times[D]}$ est donc 
r\'eflexif sur $X$. Par ailleurs, les isomorphismes 
$\cL_{|U\times[D]}\cong\cE_{|U\times[D]}\cong\cO_{U}(D_{|U})$ sur $U \times [D]$ entra\^inent, avec la proposition~\ref{prop:Refl}, que 
$\cL_{|X\times [D]}$ est isomorphe \`a $\cO_X(D)$. La restriction du faisceau $\cL$ \`a la fibre g\'en\'erique est l'extension r\'eflexive du faisceau inversible 
correspondant. Puisque $\cL$ est localement libre exactement sur $V\times\Picard(X)$, l'ouvert d'inversibilit\'e du faisceau r\'eflexif $\cO_X(D)$ est exactement $V$ pour 
tout $[D]\in\Omega$ d'apr\`es le lemme~\ref{lem:ramero}. Puisque $\Omega$ engendre $\Picard(X)$, 
on en d\'eduit que tout diviseur de Weil sur $X$ alg\'ebriquement \'equivalent \`a z\'ero est localement libre sur $V$. 
De plus, pour tout point $x\in X\setminus V$ il existe $[D]\in\Picard(X)$ tel que $\cO_X(D)$ n'est pas localement libre en $x$.

\par{\it Troisi\`eme \'etape. }
Soient $D_1,\ldots,D_r$ des g\'en\'erateurs du groupe de N\'eron--Severi de $X$, et $V_1,\ldots,V_r$ les ouverts o\`u ces diviseurs de Weil sont localement principaux. 
Tout diviseur de Weil $D$ sur $X$ est donc localement principal sur $V\cap V_1\cap\ldots\cap V_r$, et, pour tout point ext\'erieur, il existe 
un diviseur de Weil sur $X$ qui n'y est pas localement principal. De mani\`ere \'equivalente, cet ouvert est le lieu o\`u l'extension r\'eflexive de tout fibr\'e en droites 
sur $U$ est localement libre: d'apr\`es le lemme~\ref{lem:extension}, c'est le lieu de factorialit\'e locale $X^\fact$ de $X$.
\end{proof}

\subsection{}\label{sub:finesse}

Dans le cas o\`u la vari\'et\'e normale $X$ est projective sur $\ik$, la d\'emonstration pr\'ec\'edente montre en particulier que le lieu de factorialit\'e locale $X^\text{fact}$
 est contenu dans le lieu de d\'efinition de l'application d'Albanese $V_X$ (la troisi\`eme \'etape donne d'ailleurs une description tr\`es pr\'ecise de $X^\text{fact}$). 
Si $X$ est localement factorielle, l'application d'Albanese $\alpha_X \colon X \dashrightarrow \Alb(X)$ est donc d\'efinie sur $X$ toute enti\`ere. 

On en d\'eduit le r\'esultat suivant pour une vari\'et\'e non n\'ecessairement projective (cf~\ref{directproof}):

\begin{cor}\label{cor:defalb}
Soit $X$ une $\ik$-vari\'et\'e localement factorielle. Alors toute application rationnelle $X \dashrightarrow A$ de $X$ vers une vari\'et\'e ab\'elienne $A$ est d\'efinie sur $X$ toute enti\`ere.
\end{cor}

\section{Lieu de factorialit\'e locale d'un produit}\label{factprod}

\begin{theorem}\label{theo:produit}
Soient $X$ et $Y$ deux vari\'et\'es normales. Le lieu de factorialit\'e locale de $X\times Y$ est le produit des lieux de factorialit\'e locale de $X$ et de $Y$.
\end{theorem}

\begin{proof}
La question \'etant locale, on peut, comme dans la d\'emonstration du th\'eor\`eme~\ref{theo:gros1}, supposer que $X$ et $Y$ sont projectives sur $\ik$.

Fixons deux $k$-points lisses $a\in X$ et $b \in Y$. Nous notons $U_X$ et $U_Y$ les ouverts de lissit\'e de $X$ et $Y$, et $V_X$ et $V_Y$ les ouverts de d\'efinition de leurs
 applications d'Albanese 
$\alpha_X$ et $\alpha_Y$. Rappelons que $X\times_\ik Y$ est irr\'eductible et
normale (d'apr\`es~\cite[6.14.3]{egaIV2}), $U_{X\times Y}=U_X\times U_Y$, et $\Alb(X\times Y)\cong\Alb(X)\times\Alb(Y)$ de telle sorte que $V_{X\times Y}=V_X\times V_Y$.

Soit $L\in\Pic(U_{X\times Y})$, que l'on d\'ecompose sous la forme:
\[
L\cong p_X^*L_1\otimes p_Y^*L_2\otimes M,
\]
o\`u $p_X,p_Y$ sont les projections respectives de $U_X\times U_Y$ sur $U_X$ et $U_Y$, $L_1\in\Pic(U_X)$, $L_2\in\Pic(U_Y)$ et $M\in\Pic(U_{X\times Y})$ est tel que les restrictions $M_{|\{a\}\times Y}$ et $M_{|X\times\{b\}}$ sont triviales. Au diviseur de Weil associ\'e \`a l'extension r\'eflexive de $M$ \`a $U_X\times Y$ correspond un morphisme point\'e naturel $f_M\colon U_X\to\Picard(Y)$, donc un morphisme 
$\phi_M\colon\Alb(X)\to\Picard(Y)$ d'apr\`es la proposition~\ref{picardWeil}. Par la propri\'et\'e universelle de $\Picard(Y)$, ce morphisme correspond \`a un fibr\'e en droites $N$ sur $\Alb(X)\times\Alb(Y)$. Tous ces morphismes apparaissent dans le diagramme commutatif:
\begin{equation*}
\xymatrix@C=35pt{
U_X \times U_Y \ar[r]^-{f_M\times\id} \ar[rd]_-{{\alpha_X}_{|U_X}\times\id} & \Picard(Y) \times U_Y \ar[r]^-{\id\times{\alpha_Y}_{|U_Y}} & \Picard(Y) \times \Alb(Y) \\
& \Alb(X) \times U_Y \ar[r]^-{\id\times{\alpha_Y}_{|U_Y}} \ar[u]_-{\phi_M\times\id} & \Alb(X) \times \Alb(Y) \ar[u]_-{\phi_M\times\id}
}
\end{equation*}
de telle sorte que:
\begin{align*}
M&=(f_M\times\id)^*(\id\times{\alpha_Y}_{|U_Y})^*\cP_{\Alb(Y)},\\
N&=(\phi_M\times\id)^*\cP_{\Alb(Y)},
\end{align*}
o\`u $\cP_{\Alb(Y)}$ d\'esigne le fibr\'e de Poincar\'e sur $\Alb(Y)\times\Picard(Y)$ (apr\`es permutation des facteurs). On r\'ealise ainsi $M$ comme image r\'eciproque du fibr\'e inversible $N$ sur $\Alb(X)\times\Alb(Y)$, ce qui implique que l'extension r\'eflexive de $M$ \`a $X\times Y$ est localement libre sur $V_X\times V_Y$.

Notons par ailleurs $W_1$ et $W_2$ les ouverts respectifs de $X$ et $Y$ sur lesquels les extensions r\'eflexives de $L_1$ et $L_2$ sont localement libres. 
L'extension r\'eflexive du fibr\'e $L$ est donc localement libre sur l'intersection de $V_X\times V_Y$ avec $W_1\times W_2$. Mais, d'apr\`es~\ref{sub:finesse}, cette 
intersection contient $X^{\text{fact}}\times Y^{\text{fact}}$. D'apr\`es le lemme~\ref{lem:extension}, cela signifie que $(X\times Y)^{\text{fact}}$ contient $X^{\text{fact}}\times Y^{\text{fact}}$.
La r\'eciproque est imm\'ediate.
\end{proof}

\begin{cor}\label{cor:produit}
Le produit de vari\'et\'es localement factorielles est une vari\'et\'e localement factorielle.
\end{cor}

\begin{rem}\label{rem:produit}
Si $X$ et $Y$ sont des vari\'et\'es normales et projectives sur $\ik$, l'argument utilis\'e dans la d\'emonstration de la proposition~\ref{theo:produit} montre 
 que $\Cl(X\times Y)\cong \Cl(X)\times\Cl(Y)\times\Hom(\Alb(X),\Picard(Y))$. On peut en d\'eduire le r\'esultat suivant, dont nous 
ne connaissons pas de d\'emonstration directe: 

\begin{cor} 
Soient $A$ et $B$ deux alg\`ebres de type fini sur 
un corps alg\'ebriquement clos $\ik$. Si $A$ et $B$ sont factorielles, alors leur produit tensoriel $A \otimes_\ik B$ est encore factoriel.
\end{cor}

\begin{proof}
 On peut \`a nouveau plonger 
$\Spec A$ et $\Spec B$ dans des vari\'et\'es projectives normales $X$ et $Y$. Puisque $\Cl(A)=0$, on a que
$\Cl(X)$ est de type fini: il est engendr\'e par les diviseurs de Weil $D_1,\ldots,D_r$ contenus dans le bord $X\setminus \Spec A$. 
On en d\'eduit que la vari\'et\'e d'Albanese $\Alb(X)$ est triviale. 
De m\^eme, $\Cl(Y)$ est de type fini, engendr\'e par des diviseurs $E_1,\ldots,E_s$, et $\Alb(Y)$ est triviale. On obtient ainsi que $\Cl(X \times Y)$ est exactement la somme
 $\Cl(X) \oplus \Cl(Y)$. Mais 
$\Cl(\Spec (A \otimes_\ik B))$ est un quotient de:
\[
\Cl(X \times Y) / \langle D_1 \times Y,\ldots, D_r \times Y, X \times E_1,\ldots,X \times E_s \rangle.
\]
Puisque $\Cl(X) /\langle D_1,\ldots,D_r\rangle$ et $\Cl(Y) / \langle E_1,\ldots,E_s \rangle$ sont triviaux, il en est de m\^eme de 
$\Cl(\Spec(A \otimes_\ik B))$, et $A \otimes_\ik B$ est donc factoriel.
\end{proof}
\end{rem}

\begin{rem}
 Le r\'esultat est faux si l'on ne suppose plus $\ik$ alg\'ebriquement clos. En effet, on v\'erifie facilement par descente galoisienne que le quotient 
$\IA^2_\IQ/(\IZ/3\IZ)$ du plan affine par un automorphisme lin\'eaire d'ordre $3$ d\'efinit une 
vari\'et\'e $X$  sur le corps $\IQ$ des rationnels qui est localement factorielle, tandis que la vari\'et\'e $X_{\bar{\IQ}}$ qui s'en d\'eduit par extension des scalaires \`a 
$\bar{\IQ}$ ne l'est plus.
\end{rem}

\section{G\'en\'eralisation \`a la $\IQ$-factorialit\'e}\label{Qfact}

On \'etend dans cette section les r\'esultats pr\'ec\'edents aux vari\'et\'es localement $\IQ$-factorielles. Cela n\'ecessite une \'etude plus fine du comportement de la formation 
de l'extension r\'eflexive de $\cE$ \`a $X\times\Picard(X)$ vis \`a vis des 
changements de base $T \lra \Picard(X)$.
\begin{theorem}\label{theo:gros2}
Soit $X$ une vari\'et\'e d\'efinie sur un corps alg\'ebriquement clos. Alors le lieu de $\IQ$-factorialit\'e 
de $X$ est ouvert. 
\end{theorem}

\begin{proof} Il suffit \`a nouveau de d\'emontrer le r\'esultat lorsque $X$ est normale et projective sur $\ik$. Soient $a$ un point rationnel de l'ouvert de lissit\'e $U$ de $X$, 
$\Alb(X)$ la vari\'et\'e d'Albanese 
associ\'ee, $V$ l'ouvert de d\'efinition de l'application d'Albanese $X \dashrightarrow \Alb(X)$ et $\cE$ le fibr\'e sur $V \times \Picard(X)$ image inverse 
par $\alpha_X \times \id$ du fibr\'e de Poincar\'e sur $\Alb(X) \times \Picard(X)$.

Lorsque $\ik$ est la cl\^oture alg\'ebrique d'un corps fini, $\Alb(X)$ est de torsion. Cela signifie que tout diviseur de Weil alg\'ebriquement \'equivalent \`a z\'ero 
est $\IQ$-Cartier. Soient $D_1,\ldots,D_r$ des g\'en\'erateurs du groupe de N\'eron--Severi de $X$, et $V'_1,\ldots,V'_r$ les ouverts sur lesquels ces diviseurs sont 
localement $\IQ$-Cartier. 
Le lieu de $\IQ$-factorialit\'e est alors l'intersection $V'_1 \cap \ldots \cap V'_r$.

Supposons maintenant que $\ik$ n'est pas la cl\^oture alg\'ebrique d'un corps fini.

On construit tout d'abord une partition finie de $\Picard(X)$ en sous-sch\'emas localement ferm\'es lisses irr\'eductibles $\Picard(X)=\coprod P_i$ tels que, pour tout $i$, la formation de l'image directe par l'immersion 
ouverte $U \to X$ de la restriction 
$\cE_i$ de $\cE$ \`a $U \times P_i$ commute \`a tout changement de base $T \to P_i$.

On a vu, dans la deuxi\`eme \'etape de la d\'emonstration du th\'eor\`eme~\ref{theo:gros1}, que la r\'eflexivit\'e de l'extension 
r\'eflexive $\cL=(i \times \id)_\ast \cE$ assure l'existence d'un recouvrement ouvert $X \times \Picard(X)=\bigcup W_i$ 
tel que sur chaque $W_i$ la restriction $\cL_{|W_i}$ de $\cL$ \`a $W_i$ s'ins\`ere dans une suite exacte:
\[
0 \to \cL_{|W_i} \to \cF_{W_i} \to \cG_{W_i} \to 0,
\]
o\`u $\cF_{W_i}$ est localement libre et $\cG_{W_i}$ sans torsion, 
tandis que $\cG_{W_i}$ s'ins\`ere dans une suite exacte:
\[
0\to\cG_{W_i}\to\cH_{W_i}\to\cN_{W_i}\to 0,
\]
o\`u $\cH_{W_i}$ est localement libre, et que le th\'eor\`eme de platitude g\'en\'erique fournit un ouvert $\Omega$ de $\Picard(X)$ tel que la formation 
de l'image directe par $U \times \Omega \to X \times \Omega$ de la restriction de $\cE$ \`a $U \times \Omega$ commute \`a tout changement de base $T \to \Omega$. 

Soit alors $Y$ le sch\'ema r\'eduit associ\'e \`a une composante irr\'eductible du sous-sch\'ema ferm\'e $\Picard(X) \setminus \Omega$.
On obtient de la m\^eme mani\`ere, en consid\'erant la restriction du faisceau localement libre $\cE$ au produit de $U$ et de l'ouvert de lissit\'e 
de $Y$, un ouvert lisse de $Y$ ayant la propri\'et\'e attendue. On en d\'eduit par r\'ecurrence 
noeth\'erienne la partition $\Picard(X)=\coprod P_i$ annonc\'ee.

En particulier, si $\cL_i$ d\'esigne l'extension r\'eflexive de $\cE_i$ \`a $X \times P_i$, le faisceau divisoriel $\cO_X(D)$ associ\'e \`a un point $[D]$ de $\Picard(X)$ 
appartenant \`a
 l'image de $P_i$ s'obtient comme la restriction ${\cL_i}_{|X \times {[D]}}$ du faisceau $\cL_i$ \`a $X \times [D]$.

Consid\'erons le sous-ensemble:
\[
\Sigma=\{(x,[D]) \in X(\ik) \times \Picard(X)(\ik)| \cO_X(D) \ \text{est inversible au point } x\}
\]
\noindent de l'ensemble  des $\ik$-points du produit $X \times \Picard(X)$. 
C'est une partie constructible, puisque sa trace sur chaque partie $X \times P_i$ est ouverte: en effet, le faisceau r\'eflexif $\cL_i$ est, d'apr\`es le th\'eor\`eme de 
Ramanujam-Samuel (voir la deuxi\`eme \'etape de la d\'emonstration de~\ref{theo:gros1}), inversible sur un ouvert de la 
forme $V_i \times P_i$, o\`u $V_i$ est un ouvert de $X$, et le lemme~\ref{lem:ramero} montre alors que $\Sigma \cap (X \times P_i)$ est exactement $V_i \times P_i$.

Consid\'erons maintenant la projection de $\Sigma$ sur $X$. Sa fibre $A_x$ au-dessus d'un point ferm\'e $x \in X$ forme un sous-groupe constructible de 
l'ensemble des $\ik$-points 
de la vari\'et\'e ab\'elienne $\Picard(X)$. C'est donc un sous-groupe ferm\'e de $\Picard(X)$.

L'ouvert de d\'efinition $V$ de l'application d'Albanese $X \dashrightarrow \Alb(X)$ est le lieu des points auxquels tout diviseur de Weil alg\'ebriquement 
\'equivalent \`a z\'ero est inversible. C'est donc exactement le lieu des points $x$ de $X$ tels que $A_x=\Picard(X)$. 

Nous montrons que c'est aussi le lieu o\`u tout diviseur 
de Weil alg\'ebriquement \'equivalent \`a z\'ero est $\IQ$-Cartier. Soit en effet $x$ un point ferm\'e de $X \setminus V$. La vari\'et\'e ab\'elienne $\Picard(X)/A_x$ n'est alors pas 
triviale, et contient donc un $\ik$-point qui n'est pas de torsion (d'apr\`es~\cite[Theorem 10.1]{freyjarden}). Un rel\`evement de ce point \`a $\Picard(X)$ d\'efinit un diviseur 
de Weil alg\'ebriquement \'equivalent \`a z\'ero qui n'est pas $\IQ$-Cartier en $x$.

Soient des g\'en\'erateurs $D_1,\ldots,D_r$ du groupe de N\'eron--Severi de $X$, et $V'_1,\ldots,V'_r$ les ouverts sur lesquels ces diviseurs sont localement $\IQ$-Cartier. 
Le lieu de $\IQ$-factorialit\'e est alors l'intersection $V \cap V'_1 \cap \ldots \cap V'_r$.
\end{proof}

\begin{rem}
On peut donner une d\'emonstration beaucoup plus courte de ce r\'esultat lorsque $X$ est une vari\'et\'e projective normale complexe. Reprenons les notations de 
la d\'emonstration du th\'eor\`eme~\ref{theo:gros1}. L'intersection $\bigcap_{m \geqslant 1} m^{-1} \Omega$ dans $\Picard(X)$ des images 
inverses de l'ouvert $\Omega$ par la multiplication $m \colon \Picard(X) \to \Picard(X)$ est dense, comme intersection d\'enombrable d'ouverts denses.
Il existe donc un diviseur de Weil $D$ dont tout multiple $mD$ est inversible exactement sur $V$. On peut alors conclure comme pr\'ec\'edemment.
\end{rem}

On obtient une g\'en\'eralisation du corollaire~\ref{cor:defalb} (voir remarque~\ref{directproof}):

\begin{cor}
 Soit $X$ une vari\'et\'e localement $\IQ$-factorielle d\'efinie sur un corps alg\'ebriquement clos qui n'est pas la cl\^oture alg\'ebrique d'un corps fini. Alors toute application rationnelle $X \dashrightarrow A$ de $X$ vers une vari\'et\'e ab\'elienne est d\'efinie sur $X$ 
toute enti\`ere.
\end{cor}

\begin{rem}
La conclusion du corollaire pr\'ec\'edent est \'egalement vraie pour toute vari\'et\'e complexe $X$ projective normale \`a singularit\'es rationnelles 
(voir~\cite[Lemma 8.1]{kawamata}; il suffit que $X$ admette une r\'esolution $f \colon Y \to X$ v\'erifiant $R^1f_\ast \cO_Y = 0$). Remarquons qu'il existe des singularit\'es 
localement $\IQ$-factorielles mais non rationnelles, comme l'origine du c\^one affine au-dessus d'une hypersurface lisse dans $\IP^4$ de degr\'e au moins $5$.
\end{rem}

\begin{theorem}
Soient $X$ et $Y$ deux vari\'et\'es alg\'ebriques d\'efinies sur un corps alg\'ebriquement clos $\ik$ qui n'est pas la cl\^oture alg\'ebrique d'un corps fini. Alors le lieu de 
$\IQ$-factorialit\'e du produit direct $X \times Y$ est le produit des lieux de $\IQ$-factorialit\'e de $X$ et $Y$. 
\end{theorem}

\begin{proof}
 La d\'emonstration est analogue \`a celle du th\'eor\`eme~\ref{theo:produit}, le point \'etant que sous l'hypoth\`ese faite sur $\ik$ le lieu de d\'efinition du morphisme d'Albanese
co\"incide avec le lieu de $\IQ$-factorialit\'e des diviseurs alg\'ebriquement \'equivalents \`a z\'ero sur $X$.
\end{proof}

\begin{rem}
 Lorsque $\ik$ est la cl\^oture alg\'ebrique d'un corps fini, ce r\'esultat tombe en d\'efaut. Consid\'erons en effet une cubique plane lisse $E \subset \PP^2$, $C$ 
le c\^one affine correspondant, et $\overline C$ la fermeture projective de ce c\^one dans $\PP^3$. Le groupe des classes de diviseurs de Weil de ce c\^one projectif est 
alors isomorphe \`a $\Cl(E)$, autrement dit \`a $\hat E \times \IZ$, 
tandis que son groupe de Picard est isomorphe \`a $\IZ$. 
Puisque $\hat E(\ik)$ est de torsion,
$\overline C$ est une vari\'et\'e localement $\IQ$-factorielle.
On a par ailleurs $\Alb(\overline C)=E$, donc 
le groupe des classes de diviseurs $\Cl(\overline C \times_\ik \overline C)$
 s'identifie \`a:
\[
\Cl(\overline C) \times \Cl(\overline C) \times \Hom(E,\Pic^0(E))\simeq \hat E^2 \times \IZ^2 \times \Hom(E,\hat E),
\]
 et $\Pic(\overline C \times \overline C)$ \`a 
$\Pic(\overline C) \times \Pic(\overline C) \simeq \IZ^2$. Puisque $\Hom(E,\hat E)$ contient 
$\NS(E) \simeq \IZ$, le produit $\overline C \times_\ik \overline C$ n'est pas localement $\IQ$-factoriel. Le diviseur de Weil sur $\overline C \times_\ik \overline C$ associ\'e \`a la polarisation principale 
est le diviseur $\Delta$  d\'efini 
comme la cl\^oture de l'image inverse
 de la diagonale $E \subset E \times E$: c'est un exemple de diviseur 
dont aucun multiple n'est localement principal.
\end{rem}

\subsection{} Soit $X$ une $\ik$-vari\'et\'e normale. Il r\'esulte de la d\'emonstration du th\'eor\`eme~\ref{theo:gros2} que $X$ admet une partition en 
sous-espaces localement ferm\'es ${X = \coprod Z_i}$
sur lesquels le groupe des classes de diviseurs $\Cl(\cO_{X,x})$ ne d\'epend pas du point $x \in Z_i$. Nous renvoyons \`a Benoist~\cite[Theorem~8]{Benoist} 
pour plus de d\'etails.

Cette d\'emonstration a \'egalement la cons\'equence suivante: si $\ik$ n'est pas la cl\^oture alg\'ebrique d'un corps fini, alors le groupe des classes de diviseurs 
$\Cl(\cO_{X,x})$ d'un point $x \in X$ est de torsion si et seulement s'il est fini. On en d\'eduit le:

\begin{cor}
 Soit $A$ une alg\`ebre locale normale essentiellement de type fini sur un corps alg\'ebriquement clos qui n'est pas la cl\^oture alg\'ebrique d'un corps fini. 
Alors le groupe des classes de diviseurs $\Cl(A)$ de $A$ est de torsion si et seulement s'il est fini.
\end{cor}

\bibliographystyle{amsplain}
\bibliography{Biblio}

\end{document}